\documentclass[12pt]{amsart}
\usepackage{amscd,amsmath,amsthm,amssymb,graphics}
\usepackage{multicol}
\usepackage{amssymb}
\usepackage{graphicx}
\usepackage{tikz}
\usepackage{hyperref}
\usepackage{graphicx}
\usepackage{tikz}
\usepackage{enumerate}
\usepackage{amsfonts,amssymb,amscd,amsmath,enumerate,verbatim}

\definecolor{verylight}{gray}{0.97}
\definecolor{light}{gray}{0.9}
\definecolor{medium}{gray}{0.85}
\definecolor{dark}{gray}{0.6}

\def\NZQ{\Bbb}               

\def\QQ{{\NZQ Q}}
\def\ZZ{{\NZQ Z}}

\def\KK{{\NZQ K}}
\def\frk{\frak}               

\def\Phi{{\frk n}}
\def\Phi{{\frk N}}
\def\opn#1#2{\def#1{\operatorname{#2}}} 
\opn\chara{char} \opn\length{\ell} \opn\pd{pd} \opn\rk{rk}
\opn\projdim{proj\,dim} \opn\injdim{inj\,dim} \opn\rank{rank}
\opn\depth{depth} \opn\grade{grade} \opn\height{height}
\opn\embdim{emb\,dim} \opn\codim{codim}

\opn\Tr{Tr} \opn\bigrank{big\,rank}
\opn\superheight{superheight}\opn\lcm{lcm}
\opn\trdeg{tr\,deg}
\opn\reg{reg} \opn\lreg{lreg} \opn\ini{in} \opn\lpd{lpd}
\opn\size{size}\opn\bigsize{bigsize}
\opn\cosize{cosize}\opn\bigcosize{bigcosize}
\opn\sdepth{sdepth}\opn\sreg{sreg}
\opn\link{link}\opn\fdepth{fdepth} \opn\trdeg{trdeg} \opn\mod{mod}
\opn\Conv{Conv}
\opn\div{div} \opn\Div{Div} \opn\cl{cl} \opn\Cl{Cl}
\opn\Spec{Spec} \opn\Supp{Supp} \opn\supp{supp} \opn\Sing{Sing}
\opn\Ass{Ass} \opn\Min{Min}\opn\Mon{Mon} \opn\dstab{dstab} \opn\astab{astab}
\opn\Syz{Syz}
\opn\Ann{Ann} \opn\Rad{Rad} \opn\Soc{Soc} \opn\Aut{Aut}
\opn\Im{Im} \opn\Ker{Ker} \opn\Coker{Coker} \opn\Am{Am}
\opn\Hom{Hom} \opn\Tor{Tor} \opn\Ext{Ext} \opn\End{End}
\opn\Aut{Aut} \opn\id{id}

\opn\nat{nat}
\opn\pff{pf}
\opn\Pf{Pf} \opn\GL{GL} \opn\SL{SL} \opn\mod{mod} \opn\ord{ord}
\opn\Gin{Gin} \opn\Hilb{Hilb}\opn\sort{sort}
\opn\S{S} \opn\dim{dim} \opn\supp{supp}\opn\trdeg{trdeg}\opn\sort{sort}
\opn\aff{aff} \opn\relint{relint} \opn\st{st}
\opn\lk{lk} \opn\cn{cn} \opn\core{core} \opn\vol{vol}
\opn\link{link} \opn\star{star}\opn\lex{lex}
\opn\Ehr{Ehr}\opn\conv{Conv}
\opn\gr{gr}

\def\pot#1#2{#1[\kern-0.28ex[#2]\kern-0.28ex]}

\opn\dirlim{\underrightarrow{\lim}}
\opn\inivlim{\underleftarrow{\lim}}
\def\Implies{\ifmmode\Longrightarrow \else
        \unskip${}\Longrightarrow{}$\ignorespaces\fi}
\def\implies{\ifmmode\Rightarrow \else
        \unskip${}\Rightarrow{}$\ignorespaces\fi}
\def\iff{\ifmmode\Longleftrightarrow \else
        \unskip${}\Longleftrightarrow{}$\ignorespaces\fi}

\let\:=\colon
\newtheorem{Theorem}{Theorem}[section]
 \newtheorem{Lemma}[Theorem]{Lemma}
 \newtheorem{Corollary}[Theorem]{Corollary}
 \newtheorem{Proposition}[Theorem]{Proposition}
 \newtheorem{Remark}[Theorem]{Remark}
 
 \newtheorem{Example}[Theorem]{Example}
 
 \newtheorem{Definition}[Theorem]{Definition}

 \newtheorem*{Theorem2}{Theorem}

\let\epsilon\varepsilon
\let\kappa=\varkappa
\def\qed{\ifhmode\textqed\fi
      \ifmmode\ifinner\quad\qedsymbol\else\dispqed\fi\fi}
\def\textqed{\unskip\nobreak\penalty50
       \hskip2em\hbox{}\nobreak\hfil\qedsymbol
       \parfillskip=0pt \finalhyphendemerits=0}
\def\dispqed{\rlap{\qquad\qedsymbol}}

\opn\dis{dis}
\def\pnt{{\raise0.5mm\hbox{\large\bf.}}}

\opn\Lex{Lex}

\begin{document}
 \title {Gorenstein property for phylogenetic trivalent trees}

 \author {Rodica Dinu, Martin Vodi\v{c}ka}

\address{Faculty of Mathematics and Computer Science, University of Bucharest, Str. Academiei 14, 010014 Bucharest, Romania}
\email{rdinu@fmi.unibuc.ro}

\address{Max Planck Institute for Mathematics in Sciences, Inselstrasse 22, 041 03 Leipzig, Germany}
\email{vodicka@mis.mpg.de}


 \begin{abstract}
We study the Gorenstein property for phylogenetic group-based models. We prove that for the groups $\mathbb Z_3$ and $\mathbb Z_2\times \mathbb Z_2$ and trivalent trees the associated polytopes are always Gorenstein extending the results of Buczy\'nska and Wi\'sniewski for the group $\mathbb Z_2$.
 \end{abstract}

\thanks{}
\subjclass[2010]{52B20, 14M25}
\keywords{group-based model, phylogenetic trivalent tree, Gorenstein polytope}

 \maketitle

\section*{Introduction}
\indent Phylogenetics is a science that models evolution and describes mutations in this process \cite{fels}. This topic reveals many connections to several branches of mathematics such as algebraic geometry \cite{erss}, \cite{bw}, \cite{mateusz}, and combinatorics \cite{mateusz_degree}, \cite{marysia_comb}.\\
 \indent One central object in phylogenetics is the tree model, which is a parametric family of probability distributions. The tree model is based on a rooted tree, a finite set of states $S$ and a family $\mathcal{M}$ of transition matrices along the edges of the tree, which is usually given by linear subspaces of all $|S|\times |S|$ matrices. When $\mathcal{M}$ is a proper subspace of matrices, the models reflect symmetries among the elements of $S$. These are usually encoded by the action of a finite group $G$ on $S$. This fact allows us to regard $\mathcal{M}$ as the space of $G$-invariants. Such models are called \textit{equivariant} \cite{draisma1}. If $G$ is the trivial group, we obtain the general Markov model, which, from an algebraic-geometry perspective, is linked to secant varieties of Segre products, see \cite{erss}. A {\it group-based} model is a tree model where the set $S$ is a group which acts on itself and parameters are $G$-invariant. \\
\indent Motivated by biology, we consider the algebraic variety $X(T, G)$ associated to a tree $T$, describing the evolution of species, and an abelian group $G$, distinguishing a model of evolution. The construction of this variety is described in \cite{erss}. For an arbitrary model, this variety is not necessarily toric. However, by  \cite{evans}, \cite{sz} and \cite{ss}, group-based models allow a monomial parametrization, thus this variety is toric. We are interested in investigating geometric properties of these algebraic varieties, following the path initiated by Sturmfels and Sullivant \cite{ss}. The construction of the defining polytope of the toric variety corresponding to a group-based model depends on the tree and on the group. There is an algorithm for finding the polytope, which is presented in \cite{mateusz} and implemented in \cite{marysia}. The biological motivation for considering variety $X(T,G)$ can be consulted in \cite{biocomp} and \cite{erss}.
 The 3-parameter Kimura model \cite{kimura} is the most biologically meaningful group-based model, since in this case $S=\{A,C,G,T\}$. Thus the basis elements correspond to the nucleotides of DNA: adenine (A), cytosine (C), guanine (G) and thymine (T). The action of the group $G=\ZZ_2\times \ZZ_2$ is actually the Watson-Crick complementarity. This model is also very interesting from a mathematical perspective. Recent results may be found e.g. in \cite{cs1}, \cite{cs2}, \cite{csm}, \cite{kaie}, \cite{mateusz2}, \cite{me}.\\
 \indent Not all group-based models give rise to normal toric varieties and there is no known a complete classification of the models for which the associated toric varieties are normal. In Section \ref{section_normality}, Proposition \ref{notnormal}, we show that for an abelian group $G$ such that $|G|=2k, k\geq 3$, and the tripod $T$, the projective algebraic variety $X(T,G)$ representing the model is not projectively normal, extending the result of \cite{marysia_mm}. In Section \ref{section_facets} we give the facet description of the defining polytope of $X(T, \ZZ_3)$, where $T$ is the $m$-claw tree, in Theorem \ref{facetsZ3}.\\
 In Section \ref{section_fiberproduct}, we investigate the toric fiber product of two Gorenstein polytopes of the same Gorenstein index and we obtain, in Theorem \ref{fibGor}, the following result:
\begin{Theorem2}
 Let  $P_1, P_2$ be two Gorenstein polytopes with the same Gorenstein index $k$.  Let $\pi_i:P_i\rightarrow\mathbb R^n, (i=1,2) $ be projections such that $\pi_1(P_1)=\pi_2(P_2)=\Delta_n$, where $\Delta_n$ is the standard simplex. Let $p_1,p_2$ be the unique interior lattice points of $kP_1$ and $kP_2$ respectively. Assume that $\pi_1(p_1)=\pi_2(p_2)$. Then the fiber product $P=P_1\times_{\Delta_n}P_2$ is also Gorenstein of index $k$.
\end{Theorem2}

 In a seminal paper of Buczy\'nska and Wi\'sniewski \cite{bw}, in Theorem 2.15, the authors showed that $X(T, \mathbb{Z}_2)$, the variety associated to any trivalent tree and the group $\mathbb{Z}_2$, is Gorenstein Fano variety of index $4$. In Section \ref{section_gorenstein}, we extend this result to other small groups and we give the main result of this paper:
\begin{Theorem2}
Let $\mathcal{T}$ be a trivalent tree. Then the polytope associated to $\mathcal{T}$ and the group $\mathbb{Z}_{2}\times \mathbb{Z}_{2}$ is Gorenstein of index $4$ and the polytope associated to $\mathcal{T}$ and the group $\mathbb{Z}_{3}$ is Gorenstein of index $3$.
\end{Theorem2}
 We reserve the last section, Section \ref{section_z3normal}, of the paper to show the normality for $P_m$, the defining polytope of the variety $X(T, \mathbb Z_3)$, where $T$ is the $m$-claw tree, in Theorem \ref{normalz3}. So far normality was confirmed only for two other models: $G=\ZZ_2$ and 3-Kimura model \cite{martinko}.\\

\section{Normality} \label{section_normality}

Normality is an important property of lattice polytopes, see \cite{brunsg}, \cite{sturmfels}, \cite{binomialideals}, \cite{cox}. Recall that a polytope $P$ whose vertices generate lattice $L$ is {\it normal} if every point in $kP\cap L$ can be expressed as a sum of $k$ points from $kP\cap L$. The normality of polytope implies the projective normality of the associated projective toric variety.

When one wants to discuss the Gorenstein property, one usually assume that the polytopes are normal. Part of definition for being Gorenstein polytope is being normal (\cite[Section 12.5]{monomialideals}). In the case of a non-normal polytope $P$, we can still ask whether the algebra $\mathbb C[P]$ is Gorenstein, but the criteria to verify that are more difficult, see \cite{hoa}. In this article, we will study the Gorenstein property only for normal polytopes. Thus, we would like to know which polytopes associated to the phylogenetic tree models are normal.

Unfortunately, for many models we do not know whether the associated polytopes are normal. We give a short summary of known results and then we prove new results.
By {\it $m$-claw tree} we mean the graph having $m$ vertices, one of degree $m-1$ and all the other are leaves. In particular, when $m=3$, we call this graph a {\it tripod}.

Micha{\l}ek showed in \cite[Lemma 5.1]{mateusz} that, for a given group, in order to check the normality for the algebraic variety for this group and any trivalent tree, it is enough to check normality for the chosen group and the tripod. More generally, if one wants to check the normality for all trees, it is enough to verify normality for claw trees.

Buczy\'nska and Wi\'sniewski \cite{bw} proved that the toric variety associated to any trivalent tree and the group $\mathbb{Z}_2$ is projectively normal. The same result holds actually for any tree by using \cite{martinko}, where the normality was proved for the $3$-Kimura model.

However, there are examples of group-based models whose associated toric variety is not normal. By \cite[Computation 4.1]{marysia_mm}, the polytope associated to the tripod and any $G\in~\{\mathbb{Z}_6, \mathbb{Z}_8, \mathbb{Z}_2\times \mathbb{Z}_2 \times \mathbb{Z}_2, \mathbb{Z}_4\times \mathbb{Z}_2\}$ gives non-normal algebraic varieties representing the model. The next result shows that, in fact, we always obtain non-normal models for the tripod and any abelian group $G$ such that $|G|=2k, k\geq 3$.

Let $P_{G,m}\subseteq \mathbb{R}^{m|G|}$ be the polytope associated to the $m$-claw tree and the group $G$. We label the coordinates of a point $x\in\mathbb R^{m|G|}$ by $x_g^j$ where $1\le j\le m$ and $g\in G$, i.e. $g$ corresponds to group element and $j$ to the edge of the tree. For any point $x\in \mathbb Z_{\ge 0}^{m|G|}$ we define its $G$-presentation as an $m$-tuple $(G_1,\dots,G_m)$ of multisets of elements of $G$. Every element $g\in G$ appears exactly $x_g^j$ times in the multiset $G_j$. We denote by $x(G_1,\dots,G_m)$ the point with the corresponding $G$-presentation. In the case where all multisets contain exactly one element we may simply say that the $G$-presentation is just an $m$-tuple of elements of $G$.

The vertex description of the polytope $P_{G,m}$ is known; see \cite{ss}, \cite{bw}, \cite{mateusz}.
We recall this description and we formulate in the language of $G$-presentations the vertex description of $P_{G,m}$.

\begin{Theorem}The vertices of the polytope $P_{G,m}$ associated to the $m$-claw tree and the finite abelian group $G$ are exactly the points $x(g_1,\dots,g_m)$ with $g_1+\dots+g_m=0$.

Let $L_{G,m}$ be the lattice generated by vertices of $P_{G,m}$. Then
 $$L_{G,m}=\{x\in \mathbb Z^{m|G|}:\sum_{g,j}x_g^j\cdot g=0,
\forall \text{ } 1\le j,j' \le m, \sum_g x_g^j=\sum_g x_g^{j'} \}.$$
where the first sum is taken in the group $G$.
\end{Theorem}

\begin{Example}
Let $G=\mathbb Z_2$, $m=3$. Then
 \begin{align*}
P_{\ZZ_2,3}&=\Conv(\{x(0,0,0), x(1,1,0), x(1,0,1), x(0,1,1)\})
\\&=\Conv(\{(1,0,1,0,1,0),(0,1,0,1,1,0),(0,1,1,0,0,1),(1,0,0,1,0,1)\}).
\end{align*}
Then the lattice generated by the vertices of the polytope $P_{\ZZ_2,3}$ is
\begin{align*}
L_{\ZZ_2,3}&=\{x\in\mathbb Z^6: 2\mid x_2+x_4+x_6,\ x_1+x_2=x_3+x_4=x_5+x_6\}.
\end{align*}
\end{Example}

\begin{Proposition} \label{notnormal}
Let $G$ be an abelian group with even cardinality and at least 6 elements. Then the polytope $P_{G,3}$ associated to the tripod and the group $G$ is not normal. Thus, the projective algebraic variety representing the model is not projectively normal.
\end{Proposition}

\begin{proof}
Suppose that $G\neq \mathbb Z_2^n$. We prove that there exist elements $g,h\in G$ such that $2g=0$ and $2h\neq 0,g$.

Suppose for contradiction that such elements do not exist. Since $G$ has even number of elements we can find an element $g$ of order 2. By assumption for every $h\in G$ we have $2h=0$ or $2h=g$. Therefore, every element of $G$ has order 2 or 4. This implies $G=\ZZ_2^k\times \ZZ_4^l$ with $k\ge 0, l\ge 1$. However, for the group $\mathbb Z_2\times \mathbb Z_4$ we are able to find such $g,h$, e.g. $g=(1,0), h=(0,1)$. All other groups in this form (except of $\mathbb Z_4$ which does not have 6 elements) contain $\mathbb Z_2\times \mathbb Z_4$ as a subgroup. In this case we may pick the same elements from the subgroup.

We pick such $g,h$ and consider the point $$p=x(\{0,g,h,g+h\},\{0,g,h,g+h\},\{0,g,-2h,g-2h\}).$$ We notice that $$2p=x(0,0,0)+x(g,g,0)+x(g,0,g)+x(0,g,g)+x(h,h,-2h)+$$ $$+x(g+h,g+h,-2h)+x(g+h,h,g-2h)+x(h,g+h,g-2h).$$ Since each summand is a vertex of $P_{G,3}$ we obtain that $p\in 4P$. Moreover, $p\in L_{G,3}$ since the sum of all elements in $G$-presentation of $p$ is 0. We show that $p$ cannot be written as a sum of 4 vertices of $P_{G,3}$.

Expressing $p$ in such way is equivalent to splitting elements from $G$-presentation into triples with sum 0 and consisting of one element from each multiset. It is easy to verify that this is not possible.  Thus, $P_{G,3}$ is not normal.

For the group $G=\mathbb Z_2^3$ we conclude in an analogous way by considering the point $$p=x(\{(0,0,0),(0,0,1),(1,0,0),(1,0,1)\},\{(0,0,0),(0,0,1),(0,1,0),(0,1,1)\},$$ $$\{(0,0,0),(0,0,1),(1,1,0),(1,1,1)\}).$$ Similarly as in the previous case, we can check that $p\in 4P_{G,3}$ but it cannot be written as the sum of 4 vertices of $P_{G,3}$. The groups $G=\mathbb Z_2^n$ contain $\mathbb Z_2^3$ as a subgroup and, by \cite[Proposition 4.2]{marysia_mm}, we obtain the non-normality for the polytope associated to the tripod and the group $G=\mathbb Z_2^n$.
\end{proof}
\begin{Remark}
Non-normality for the tripod gives non-normality for any non-trivial tree (not a path). Indeed, the polytope for any tree always contains a face which is isomorphic to the polytope for the tripod (and the same group).
\end{Remark}

\section{Facet description for the polytope associated to $\mathbb{Z}_3$} \label{section_facets}

 In this section we consider the variety $X(T, \mathbb{Z}_3)$, where $T$ is the $m$-claw tree. For this part we consider only the projection of $P_{\mathbb Z_3,m}$ on the $2m$ coordinates which correspond to non-zero elements of $\mathbb Z_3$. We denote this projection simply by $P_m$.
The polytope $P_m$ has symmetries that can be described by group actions:
 \begin{itemize}
\item action of $H_{m}=\{(h_1, h_2,\dots, h_m)\in \mathbb{Z}_3^m: h_1+h_2+\dots+h_m=0\}$: For $h\in H_{m}$ and $x\in \mathbb{R}^{2m}$ we define $(hx)_i^j=x_{i+h_j}^j$, where the sum $i+h_j$ is in $\mathbb Z_3$ and $x_0^j=1-x_1^j-x_2^j$.
    \item action of $\mathbb S_m$: For $\sigma\in S_{m}$ and $x\in \mathbb{R}^{2m}$ we define $(\sigma x)_i^{\sigma(j)}=x_{i}^j$.
    \item action of $\Aut(\mathbb Z_3)$: For $\varphi\in \Aut(\mathbb Z_3)$ and $x\in \mathbb{R}^{2m}$ we define $(\varphi x)_{\varphi(i)}^j=x_{i}^j$.
\end{itemize}

\begin{Lemma}\label{vert0}
The vertices of $P_m$ which are connected with 0 by an edge are only those with at most 3 non-zero coordinates. More precisely, these are the following:
\begin{itemize}
\item $v(j_1,j_2)$ such that $v(j_1,j_2)_1^{j_1}=1,v(j_1,j_2)_2^{j_2}=1$ and all other coordinates are equal to 0, where $1\le j_1,j_2\le m, j_1\neq j_2$.
\item $v(i;j_1,j_2,j_3)$ such that $v(i;j_1,j_2,j_3)_i^{j_1}=v(i;j_1,j_2,j_3)_i^{j_2}=v(i;j_1,j_2,j_3)_i^{j_3}=1$ and all other coordinates are equal to 0, where $1\le j_1<j_2<j_3\le m$, $i\in\{1,2\}$.
\end{itemize}
\end{Lemma}

\begin{proof}
Consider any other vertex $v$ of $P_m$. If we can write $v=w_1+w_2$ for some vertices $w_1,w_2$ then there can not be an edge between 0 and $v$.
Clearly, there must either exist two indices $j_1,j_2$ such that $v_1^{j_1}=1,v_2^{j_2}=1$ or three indices $j_1,j_2,j_3$ and $i$ such that $v_i^{j_1}=v_i^{j_2}=v_i^{j_3}=1$. In the first case we can write $v=v(j_1,j_2)+w$ and in the second $v=v(i;j_1,j_2,j_3)+w$ for some vertex $w$.
Thus, there are no other vertices connected by an edge with 0.

It is not difficult to show that the vertices from the statement are, in fact, connected by an edge with 0 but we do not need this fact so we leave the proof for the reader.
\end{proof}

The vertex description for polytopes representing group-based models is already known. On the contrary, the facet description is known only in few cases, namely in case of $G=\mathbb Z_2$ and $G=\mathbb Z_2\times \mathbb Z_2$, see \cite{h-rep}. In general, it is a difficult problem to obtain the facet description from the vertex description. In the next result, we are presenting the facet description of the polytope $P_m$ for the group $\mathbb Z_3$.

\begin{Theorem}\label{facetsZ3}
The facet description of the polytope $P_m$ is given by: $$x_i^j\ge 0, \text{ for } 1\le i\le 2,1\le j\le m,$$
$$x_1^j+x_2^j\le 1, \text{ for } 1\le j\le m,$$
$$\langle u_{a_1},x^1\rangle+\langle u_{a_2},x^2\rangle+\dots+\langle u_{a_m},x^m\rangle\ge 2-a_1-a_2-\dots-a_m,$$
$$\langle w_{a_1},x^1\rangle+\langle w_{a_2},x^2\rangle+\dots+\langle w_{a_m},x^m\rangle\ge 2-a_1-a_2-\dots-a_m,$$
for all $(a_1,a_2,\dots,a_m)\in \{0,1,2\}^m$ such that $a_1+a_2+\dots+a_m\equiv 2\pmod 3$, where $u_0=(1,2),u_1=(1,-1),u_2=(-2,-1),w_0=(2,1),w_1=(-1,1),w_2=(-1,-2)$ and $x^j=(x_1^j,x_2^j).$
\end{Theorem}

\begin{proof}
Since by acting with suitable $h\in H_m$ we can map any vertex of $P_m$ to 0 it is enough to describe all facets containing 0.

Firstly, we show that the inequality $x_1^1\ge 0$ gives us a facet of $P_m$ containing 0. This inequality holds for every vertex of $P_m$ and therefore for every point in $P_m$. On the other hand, for vertices $0,v(2,3),v(3,2)$, $v(j,1)$ for $2\le j\le m $ and $v(2;1,2,j)$ for $3\le j\le m$ equality $x_1^1=0$ holds. These $2m$ vertices span a $2m-1$-dimensional affine subspace. We conclude that $x_1^1\ge 0$ gives us a facet of $P_m$. Analogously, we can prove that $x_i^j\ge 0$ defines a facet for all $i\in\{1,2\},1\le j\le m$.

Consider the facet $F$ containing 0 given by the inequality $$A_1x_1^1+B_1x_2^1+A_2x_1^2+B_2x_2^2+\dots+A_mx_1^m+B_mx_2^m\ge 0$$
for $A_j,B_j\in\mathbb R$. Since $P_m$ is a lattice polytope we may assume that all $A_j, B_j$ are integers.

Suppose that all $A_j,B_j\ge 0$. If only one of them is non-zero we obtain the facet $x_i^j\ge 0$.  If two or more coefficients are non-zero the inequality defines the intersection of such facets and therefore it is not a facet. Thus, there are no other facets for which all $A_j,B_j\ge 0$.

We may assume that at least one coefficient is negative. Without loss of generality $A_m<0$. If not, we can just act with suitable $\sigma\in\mathbb S_m$ and $\varphi \in \Aut(\mathbb Z_3)$ to get to this case. We claim that there exists a vertex $v\in F$ with $v_1^m=1$.
 Indeed, if there was no such vertex it would imply that $F$ is a subset of the facet of $P_m$ given by $x_1^m\ge 0$, which is impossible.

By definition of $P_m$, to the vertex $v$ of $P_m$ corresponds one element $h_v\in H_m$. If we act with $h_v\in H_m$ on the facet $F$ we get a facet of $P_m$ containing 0 since it maps $v$ to 0. Last two coefficients defining $h_v(F)$ are $(-B_m,A_m-B_m)$. If $B_m>0$, then these coefficients for $h_v(F)$ are both non-positive. Thus, we may also assume that $B_m\le 0$. Without loss of generality we may assume $A_m\le B_m$, by acting with suitable $\varphi\in\Aut(\mathbb Z_3)$.

Since vertices $v(m,j)$ and $v(j,m)$ satisfy the inequality we get $B_j+A_m\ge 0$ and $A_j+B_m\ge 0$. In particular, $A_j\ge 0,B_j>0$ for all $1\le j\le m-1$. Moreover, from vertices $v(1;j_1,j_2,m)$ we get that at most one coefficient $A_j$ can be equal to 0.

Since $F$ is a facet containing 0, there must exist $2m-1$ linearly independent non-zero vertices connected to 0 by an edge lying on $F$. It is easy to check that the only possible candidates for that are vertices $v(j,m)$ and $v(i;j_1,j_2,m)$ by Lemma \ref{vert0}.

That means that $2m-1$ linearly independent equations from the following must hold:
\[
A_j=-B_m, \tag{$\ast_j$}
\]
\[
A_{j_1}+A_{j_2}=-A_m, \tag{$\ast_{j_1,j_2}$}
\]
\[
B_j=-A_m, \tag{$\diamond_j$}
\]
\[
B_{j_1}+B_{j_2}=-B_m, \tag{$\diamond_{j_1,j_2}$}
\]
for $1\le j,j_1,j_2\le m-1$.

 Since in equations $(\ast_j)$ and $(\ast_{j_1,j_2})$ appear only $m+1$ coefficients we can choose at most $m$ linearly independent from them. The same is true for $(\diamond_j)$ and $(\diamond_{j_1,j_2})$. So we have to choose $m$ equations from $(\ast)$ equations and $m-1$ from $(\diamond)$ or the other way around.

 Vertices $v(m,j)$ satisfy the inequality of the facet $F$ and therefore $B_j\ge-A_m$ which leads to $B_{j_1}+B_{j_2}\ge-2A_m>-A_m\ge-B_m$. Thus, none of the equalities $(\diamond_{j_1,j_2})$ can hold. We conclude that all $m-1$ equalities $(\diamond_j)$ hold and $B_j=-A_m$ for all $1\le j\le m-1$.

 Suppose that $A_m\neq 2B_m$. From the $(\ast)$ equalities we need to choose $m$ linearly independent and therefore at least one from $(\ast)_j$ since others do not contain the coefficient $B_m$. Without loss of generality we choose $(\ast)_1,(\ast)_2,\dots,(\ast)_k$, $1\le k\le m-1$. Thus $A_1=A_2=\dots = A_k=-B_m$. The equalities $(\ast_{j_1,j_2})$ for $j_1,j_2\le k$ can not hold since $A_m\neq 2B_m$. From the equalities $(\ast_{j_1,j_2})$ for $j_1,j_2>k$ we can choose at most $m-k-1$ linearly independent so we must choose at least one with $j_1\le k,j_2>k$. For the fixed $j_2>k$ all equalities $(\ast_{j_1,j_2})$ are the same. Hence if we choose more than one of them we need to choose different $j_2$. Without loss of generality we choose $(\ast_{1,j_2})$ for all $k+1\le j_2\le l$, where $k<l\le m-1$. We are left with the equalities $(\ast_{j_1,j_2})$ for $j_1,j_2>l$. However, we can choose at most $m-k-l-1$ linearly independent which means we have together only at most $k+l+m-k-l-1=m-1$ equalities which is a contradiction. Therefore, we obtain that $A_m=2B_m$.

 In this case all equalities $(\ast)$ have common solution $A_1=A_2=\dots =A_{m-1}=-B_m=-A_m/2$ so it does not matter which of them we choose.

 Thus, we have only one solution. After dividing by $A_1$ we obtain the facet given by
 $$x_1^1+2x_2^1+x_1^2+2x_2^2+\dots+x_1^{m-1}+2x_2^{m-1}-2x_1^m-x_2^m\ge 0.$$

 All the other facets of $P_m$ can be obtained by acting with the groups $\mathbb S_m,H_m$ and $\Aut(\mathbb Z_3)$. It is easy to check that we get exactly the facets from the statement of our theorem.
\end{proof}

With the facet description of the polytope $P_m$ we are able to prove the following theorem about normality:

\begin{Theorem}\label{normalz3}
The Polytope $P_m$ is normal for every positive integer $m$.
\end{Theorem}

The last section of this paper, Section \ref{section_z3normal}, is dedicated to the proof of Theorem \ref{normalz3}.

\section{The Toric Fiber product} \label{section_fiberproduct}

We start this section by defining {\it toric fiber products}. This concept was introduced by Sullivant, in \cite{seth}. Given a positive integer $m$, we denote $[m]= \{1,\dots, m\}$. Let $r$ be a positive integer and let $s$ and $t$ be two vectors of positive integers in $\ZZ^{r}_{>0}$. We consider the following homogeneous, multigraded polynomial rings
\[
\KK[x]=\KK[x_j^i: i\in [r], j\in [s_i]] \text{ and }\\
\KK[y]=\KK[y_k^i: i\in [r], k\in [t_i]],
\]
with the same multigrading

\[
\deg(x_j^i)=\deg(y_k^i)=\mathbf{a}^i \in \ZZ^d.
\]
We assume that there exists a vector $w\in \QQ^d$ such that $\langle w, \mathbf{a}^i\rangle =1$, for any $i$. Denote $\mathcal{A}=\{\mathbf{a}^1, \dots, \mathbf{a}^d\}$. If $I\subseteq \KK[x]$ and $J\subseteq \KK[y]$ are homogeneous ideals, then the quotients rings $R=\KK[x]/I$ and $S=\KK[y]/J$ are also multigraded rings. Let
\[
\KK[z]=\KK[z_{jk}^i: i\in [r], j \in [s_i], k\in [t_i]],
\]
and consider the ring homomorphism
\[
\phi_{I,J} : \KK[z] \rightarrow R \otimes_{\KK} S; \quad
z_{jk}^{i} \longmapsto x_j^i \otimes y_k^i.
\]
\begin{Definition}
The toric fiber product of $I$ and $J$ is
\[
I\times_{\mathcal{A}} J= \Ker(\phi_{I,J}).
\]
\end{Definition}
\begin{Lemma}\label{facets}
Let $P_1$ and $P_2$ be two polytopes. Let $\pi_i:P_i\rightarrow\mathbb R^n, (i=1,2) $ be integral projections such that $\pi_1(P_1)=\pi_2(P_2)=\Delta_n$, where $\Delta_n$ is the standard simplex. Then all facets of the toric fiber product $P_1\times_{\Delta_n}P_2$ are of the form $F_1\times_{\Delta_n}P_2$ or $P_1\times_{\Delta_n}F_2$ where $F_i$ is a facet of $P_i$.
\end{Lemma}

\begin{proof}
Let $d_1,d_2$ be the dimensions of $P_1,P_2$, respectively and $d=d_1+d_2-n$ be the dimension of $P=P_1\times_{\Delta_n}P_2$.
  For $(x_i)_{i\in [d]}=x\in \mathbb{R}^d$ and subset $S\subseteq[d]$ we denote by $x_S$ the projection of $x$ to $\mathbb R^{S}$. Let $A, B\subseteq [d]$ be sets with $|A|=d_1,|B|=d_2,|A\cap B|=n$ such that $P_1\subseteq\mathbb R^A$, $P_2\subseteq \mathbb R^B$, $P\subseteq \mathbb R^{A\cup B}$ and $\Delta_n\subseteq \mathbb R^{A\cap B}$.

  For a face $F$ of $P$, we say that $F$ is \emph{good} if $F$ is of the form $F_1\times_{\Delta_n}P_2$ or $P_1\times_{\Delta_n}F_2$ where $F_i$ is a face of $P_i$.

  Let $F$ be a face of $P$ given by $u\in \mathbb Z^d$, i.e. $\langle x,u\rangle \ge a$ for any $x\in P$ and $F=\{x\in P: \langle x,u\rangle=a\}$.

We prove that if $F$ is not good, then there exists a face $\widetilde{F}\supseteq F$ such that $\widetilde{F}$ is good. This implies that there is no facet of $P$ which is not good.

For this we will find $0\neq v\in \mathbb Z^d$ such that, for any vertex $x\in P$, $\langle x,v\rangle \ge a$ and $\langle x,u\rangle=a$ implies $\langle x,v\rangle=a$ and $v_{B\setminus A}=0$. Then the face given by $v$ and $a$ satisfies all conditions.

 Denote $$m_k=\min_{x\in P,x_{A\cap B}=e_k}\langle x_B,u_B \rangle.$$ It exists because we are taking minimum in the compact set. Take $v$ such that $$v_A=u_A+\sum_{k\in A\cap B}(m_k-u_k)e_k,\ v_{B\setminus A}=0.$$ Then for any vertex $x\in P$ such that $x_{A\cap B}=e_k$ we have
 $$\langle x,v\rangle=\langle x_A,v_A\rangle=\langle x_A,u_A\rangle+m_k-u_k\ge a, $$
 since for all $x$ with $x_{A\cap B}=e_k$ we have $\langle x,u\rangle=\langle x_A,u_A\rangle+\langle x_B,u_B\rangle-u_k\ge a $. The equality statement follows in the same way. Since $F$ was not in the form from the statement we must have $u_{A\setminus B}\neq 0$ which means that $0\neq v$ and therefore this is the $v$ which we were looking for.

 Now let $F$ be a facet of $P$. We know that $F$ is be good. Without loss of generality $F=F_1\times_{\Delta_n}P_2$. But if $F_1$ is not a facet of $P_1$ then there exists a facet $\widetilde{F_1}\supsetneq F_1$. Then $F_1\times_{\Delta_n}P_2\subsetneq \widetilde{F_1}\times_{\Delta_n}P_2$ which means that $F$ is not a facet of $P$, which is a contradiction.
\end{proof}

We give now the main result of this section.
\begin{Theorem}\label{fibGor}
Let  $P_1, P_2$ be two Gorenstein polytopes with the same Gorenstein index $k$.  Let $\pi_i:P_i\rightarrow\mathbb R^n, (i=1,2) $ be projections such that $\pi_1(P_1)=\pi_2(P_2)=\Delta_n$, where $\Delta_n$ is the standard simplex. Let $p_1,p_2$ be the unique interior lattice points of $kP_1$ and $kP_2$ respectively. Assume that $\pi_1(p_1)=\pi_2(p_2)$. Then the toric fiber product $P=P_1\times_{\Delta_n}P_2$ is also Gorenstein of index $k$. Moreover, the unique interior point $p$ of $kP$ satisfies $\varphi_1(p)=p_1,\varphi_2(p)=p_2$ where $\varphi_i:P\rightarrow P_i$ are toric fiber product maps.
\end{Theorem}

\begin{proof}
We show that in the dilated polytope $kP$, the point $p$ has distance $1$ to all facets. Let $F$ be a facet of $P$. By Lemma \ref{facets}, without loss of generality, $F=~F_1\times_{\Delta_n}~P_2$. It is clear that the point $p$ has the distance to $F$ the same as the distance of the point $p_1$ to $F_1$, which is $1$, by our assumption on $P_1$ being Gorenstein.
\end{proof}

\section{Gorenstein property for claw trees and small groups} \label{section_gorenstein}

In this section we prove our main result, Gorenstein property for polytopes associated to group $\mathbb Z_3$ or $\mathbb Z_2\times\mathbb Z_2$ and any trivalent tree.

Notice, that the polytope ${P_{G,m}}$ has a symmetry that can be described by group action of $H_{G,m}=\{(h_1, h_2,\dots, h_m)\in G^m: h_1+h_2+\dots+h_m=0\}$ as in the case of $\mathbb Z_3$.

\begin{Theorem}\label{Gore}
Let $P_{G,m}$ be a polytope associated to the $m$-claw tree ($m\ge 3)$ and the group $G\in\{\mathbb Z_2,\mathbb Z_3,\mathbb Z_2\times \mathbb Z_2\}$. Then the polytope $P_{G,m}$ is Gorenstein if and only if $(G,m)\in\{(\mathbb Z_2,3),(\mathbb Z_2,4),(\mathbb Z_3,3),(\mathbb Z_2\times\mathbb Z_2,3)\}$. The Gorenstein indices in these cases are 4,2,3,4 respectively.
\end{Theorem}

\begin{proof}
In this cases the polytopes are normal by Theorem \ref{normalz3} and \cite[Theorem 14]{martinko}. Therefore our polytope is Gorenstein if and only if there exists some $k$ such that $kP_{G,m}$ has an unique interior lattice point whose lattice distance from all facets is 1, see \cite[Section 12.5]{monomialideals}.

 Suppose that there exists a unique interior lattice point $\omega_{G,m}\in kP_{G,m}$. We claim that all coordinates of $\omega_{G,m}$ must be equal. If not we can act with a suitable $h\in H_{G,m}$ to obtain a different interior lattice point. Therefore, in the case of $G\in\{\mathbb Z_3,\mathbb Z_2\times\mathbb Z_2\}$ or $m$ even, the only candidate for $\omega_{G,m}$ is $(1,1,\dots,1)$ since it belongs to $L_{G,m}$. In the case $G=\mathbb Z_2$ and $m$ odd, $(1,1,\dots,1)$ is not a lattice point in the lattice $L_{G,m}$, but the point $(2,2,\dots,2)$ is. Thus, in all cases we can denote by $\omega_{G,m}$ the only possible candidate.
We have the inequalities for $P_{G,m}$ from \cite{h-rep} and Theorem \ref{facetsZ3}. Now we compute the lattice distances from $\omega_{G,m}$ to the facets of $P_{G,m}$ and conclude our desired result. We analyze in detail the case of $\mathbb Z_2\times \mathbb Z_2$, which is the 3-Kimura model, and the other cases are similar.

Let us denote the elements of $\mathbb Z_2\times \mathbb Z_2=\{0,\alpha,\beta,\gamma\}$. The inequalities describing the $k$-dilation of the polytope are
\[
x_g^j\ge 0,\forall g\in \mathbb Z_2\times \mathbb Z_2,1\le j\le m, \tag{$\spadesuit_1$}
\]
For all $A\subseteq\{1,2,\dots,m\}$ of odd cardinality:
\[
\sum_{ j\in A}(x_0^j+x_\alpha^j)+\sum_{ j\not\in A}(x_\beta^j+x_\gamma^j)\ge k,
\]
\[
 \sum_{ j\in A}(x_0^j+x_\beta^j)+\sum_{ j\not\in A}(x_\alpha^j+x_\gamma^j)\ge k, \tag{$\spadesuit_2$}
\]
\[
 \sum_{ j\in A}(x_0^j+x_\gamma^j)+\sum_{ j\not\in A}(x_\alpha^j+x_\beta^j)\ge k.
 \]

For the last three inequalities we know that the difference of left and right side is always even if $x\in L_{G,m}$ \cite[Lemma 2]{martinko}. This means that for lattice point $x$ its distance from the facet is the difference of both sides divided by 2. In other words this facet is given by the point $y\in L_{G,m}^\vee$ where $y_0^j=y_\alpha^j=1/2,y_\beta^j=y_\gamma^j=0$ for $j\in A$ and $y_0^j=y_\alpha^j=0,y_\beta^j=y_\gamma^j=1/2$ for $j\not\in A$.

If we plug $\omega_{G,m}=(1,1,\dots,1)$ in inequality $(\spadesuit_1)$ and we obtain that the distance is always one.
If we plug it in one of the inequalities of $(\spadesuit_2)$ on the left side we obtain:

 $$\sum_{ j\in A}((\omega_{G,m})_0^j+(\omega_{G,m})_\alpha^j)+\sum_{ j\not\in A}((\omega_{G,m})_\beta^j+(\omega_{G,m})_\gamma^j)=2|A|+2(m-|A|)=2m.$$

 Since $\omega_{G,m}\in 4P_{G,m}$ the lattice distance of $\omega_{G,m}$ to this facet is $(2m-4)/2$. This is 1 only for $m=3$. Therefore, for $m=3$ our polytope is Gorenstein and otherwise not.
\end{proof}

\begin{Corollary}\label{KimuraGor}
Let $\mathcal{T}$ be a trivalent tree. Then the polytope associated to $\mathcal{T}$ and the group $\mathbb{Z}_{2}\times \mathbb{Z}_{2}$ is Gorenstein of index $4$ and the polytope associated to $\mathcal{T}$ and  the group $\mathbb{Z}_{3}$ is Gorenstein of index $3$.
\end{Corollary}

\begin{proof}
By Theorem \ref{fibGor} and Theorem \ref{Gore}.
\end{proof}

\section{Normality of the polytope associated to $\mathbb Z_3$} \label{section_z3normal}
As in Section 2, we denote by $P_m$ the defining polytope of the variety $X(T,\mathbb Z_3)$, where $T$ is the $m$-claw tree.
This whole section is devoted to the proof of the Theorem \ref{normalz3}.

The proof uses techniques from \cite{martinko} where normality is proved for the group $\mathbb Z_2\times \mathbb Z_2$.

As in the paper \cite{martinko} it is again more convenient to work with $3m$ coordinates, i.e in this section we take $P_m\subset\mathbb R^{3m}$. We will need the facet description for the dilated polytope $kP_m$. It is easy to check from Theorem \ref{facetsZ3} that the facet description of $kP_m$ is given by the following equalities and inequalities:

$$x_i^j\ge 0, \text{ for } 0\le i\le 2,1\le j\le m,$$
$$x_0^j+x_1^j+x_2^j= k, \text{ for } 1\le j\le m,$$
$$\langle u_{a_1},x^1\rangle+\langle u_{a_2},x^2\rangle+\dots+\langle u_{a_m},x^m\rangle\ge 2k,$$
$$\langle w_{a_1},x^1\rangle+\langle w_{a_2},x^2\rangle+\dots+\langle w_{a_m},x^m\rangle\ge 2k,$$
for all $(a_1,a_2,\dots,a_m)\in \{0,1,2\}^m$ such that $a_1+a_2+\dots+a_m\equiv 2\pmod 3$, where $u_0=(0,1,2),u_1=(1,2,0),u_2=(2,0,1),w_0=(0,2,1),w_1=(1,0,2),u_2=(2,1,0)$ and $x^j=(x_0^j,x_1^j,x_2^j).$

For the $m$-tuple $A=(a_1,\dots,a_m)\in\{0,1,2\}^m$ we denote the facets of $kP_m$ corresponding to $A$ and first and second inequality by $F_1(k,A)$ and $F_2(k,A)$, respectively. For the point $x$ we denote left sides of the inequalities by $S_1(x,A)$ and $S_2(x,A)$. That means $x\in F_i(k,A)\Leftrightarrow S_i(x,A)=2k$.

We want to prove that every point $x\in kP_m\cap L_m$ decomposes to a sum of lattice points from $P_m$. It is sufficient to prove it for an image of $x$ under any of group actions described as symmetries of $P_m$.

We define a linear ordering on multisets of three real numbers with sum $k$. For two multisets $\{a,b,c\}$ and $\{a',b',c'\}$ where $a\ge b\ge c$ and $a'\ge b'\ge c'$ we say $$\{a,b,c\}\succeq \{a',b',c'\}\Leftrightarrow (a>a') \vee (a=a'\wedge b>b').$$

Consider $x\in kP_m\cap L_m$. If we order multisets $\{x_0^j,x_1^j,x_2^j\}$ we can ensure that multiset for $j=m$ is the smallest one in this ordering by acting with suitable permutation from $\mathbb S_m$.

If we denote by $g_j$ the most frequent element in $j$-th multiset from $G$-presentation of $x$ we can act by $(g_1,g_2,\dots,g_{m-1},g_1+\dots+g_{m-1})\in H_m$ to obtain a point $x$ for which $x_0^j\ge x_1^j,x_2^j$ for all $1\le j<m$.

This means that for a point $x\in kP_m\cap L_m$, without loss of generality, we may assume the following two facts:

\begin{equation}
\forall j\in\{1,2,\dots,m-1\}:\ \{x_0^j,x_1^j,x_2^j\}\succeq\{x_0^m,x_1^m,x_2^m\}.
\end{equation}
\begin{equation}
\forall j\in\{1,2,\dots,m-1\}:\ x_0^j=\max\{x_0^j,x_1^j,x_2^j\}.
\end{equation}

\begin{Definition}
Let $x\in kP_m\cap L_m$. A vertex $v$ of $P_m$ is called \emph{$x$-good} if all coordinates of the point $x-v$ are non-negative.
\end{Definition}

We keep the notation from the Lemma \ref{vert0}. However, now by $v(j_1,j_2)$ we mean the same vertex but in $\mathbb R^{3m}$. We also denote $v(0)$ the vertex whose projection is 0, i.e. $v(0)_0^j=1$ for all $1\le j\le m$. Let us consider the subset of all vertices of the polytope $P_m$ given by $V_m=\{v(0),v(j,m),v(m,j)|1\le j\le m-1\}$.

We may now start the proof.

\begin{Lemma}\label{2}
Let $x\in L_m\cap kP_m$, $i\in\{1,2\}$ and $A=(a_1,\dots,a_m)\in\{0,1,2\}^m$. Then $$S_i(x,A)\equiv 2k\pmod 3.$$

\end{Lemma}

\begin{proof}
We consider only the case $i=1$, the other case is analogous.

Since $x\in L_m$ we get
$$\sum_{j=1}^m (x^1_j+2x_2^j)\equiv 0\pmod 3.$$

Moreover it is easy to verify that $$\langle u_i,x^j\rangle\equiv x_1^j+2x_2^j+i(x_0^j+x_1^j+x_2^j)\equiv x_1^j+2x_2^j+ik \pmod 3.$$

Therefore, we obtain that
$$S_i(x,A)=\sum_{j=1}^m\langle u_{a_j},x^j\rangle\equiv\sum_{j=1}^m x_1^j+2x_2^j+k\sum_{j=1}^m a_j\equiv 2k\pmod 3. $$
\end{proof}

The following result implies the fact that it is sufficient to consider only such points $x$ for which the following condition holds:

\begin{equation}
\forall j\in\{1,2,\dots,m\},g\in G: x_0^j<k.
\end{equation}

\begin{Lemma}\label{3}
Suppose that for every positive integers $k,m$ and every $x\in kP_m\cap L_m$ such that $x_0^j<k$ for all $j$ we can write $x$ as a sum of $k$ vertices of $P_n$. Then $P_m$ is normal, for every positive integer $m$.
\end{Lemma}
\begin{proof}
Similar like in the case of the group $\ZZ_2\times \ZZ_2$, see \cite[Lemma 3]{martinko}.
\end{proof}

We prove that for every point $x\in kP_m\cap L_m$ there exists a vertex $v$ such that $x-v\in (k-1)P_m\cap L_m$. It is sufficient to consider the points which satisfy $(1)-(3)$ and we will show that we can use a vertex $v\in V_m$. The normality of $P_m$ will follow by induction. Since $x-v\in L_m$ it remains to prove that $x-v\in (k-1)P_m$. We will show this by checking inequalities from the facet description of $(k-1)P_m$.

We examine all inequalities. We split them into groups depending on the sum $a_1+\dots+a_m$.

\begin{Proposition}\label{big}
Let $x\in kP_m\cap L_m, k\ge 3$ satisfy $(1)-(3)$ and $A=(a_1,\dots,a_m)\in\{0,1,2\}^m$. \begin{itemize}
\item[a)]If $a_1+\dots+a_m\ge 8$, then $S_i(x-v,A)\ge 2(k-1)$ for any $i\in\{1,2\}$ and any $x$-good vertex $v$ of $P_m$.
\item[b)] If $a_1+\dots+a_m=5,a_m<2$, then $S_i(x-v,A)\ge 2(k-1)$ for any $i\in\{1,2\}$ and any $x$-good vertex $v$ of $P_m$.
\item[c)] If $a_1+\dots+a_m=5,a_m=2$, then $S_2(x-v,A)\ge 2(k-1)$ for any $x$-good vertex $v$ of $P_m$.
\end{itemize}
\end{Proposition}

\begin{proof}
Let $y=x-v$. It is sufficient to prove parts a) and b) when $i=1$, the other case is analogous.

We first notice the following two inequalities which hold for all $1\le j\le m-1$: $$\langle u_2,y^j\rangle=2y_0^j+y_2^j=2x_0^j+x_2^j-2v_0^j-v_2^j\ge x_0^j+x_1^j+x_2^j-2=k-2,$$
$$\langle u_1,y^j\rangle =y_0^j+2y_1^j= x_0^j+2x_1^j-v_0^j-2v_1^j\ge \frac 12 (x_0^j+x_1^j+x_2^j)-1+\frac 32 x_1^j-v_1^j\ge \frac k2-1.$$
This yields $\langle u_i,y^j\rangle\ge i(k/2-1)$. Further, in all cases we have $a_1+\dots+a_{m-1}\ge 4$. Now we can prove the desired inequality:

$$S_1(x-v,A)\ge ka_1/2+\dots+ka_{m-1}/2\ge 4(k/2-1)= 2k-4.$$

We know that $S_i(x-v,A)\equiv 2k-2\pmod 3$, which implies $S_i(x-v,A)\ge 2k-2$.

For the part c): from the fact $x_1^m\ge x_2^m$ we get an analogous inequality $$\langle u_2,y^m\rangle\ge 2y_0^j+y_1^j= 2x_0^j+x_1^j-2v_0^j-v_1^j\ge \frac 12 (x_0^j+x_1^j+x_2^j)-1+\frac 32 x_0^j-v_0^j\ge \frac k2-1.$$

Together with the fact that $a_1+\dots+a_{m-1}=3$ we obtain
 $$S_2(x-v,A)\ge ka_1/2+\dots+ka_{m-1}/2+\frac k2-1\ge 3(\frac k2-1)+\frac k2-1= 2k-4,$$
 Now we conclude our desired result in the same way as in the parts a) and b).
\end{proof}

For the cases which are not covered by the Proposition \ref{big} we use induction. We know from induction hypothesis that $S_i(x,A)\ge 2k$, for all $n$-tuples $A=(a_1,\dots,a_m)$. We observe how $S_i(x,A)\ge 2k$ changes when we subtract the vertex $v$ from $x$.

\begin{Lemma}\label{small}
Let $x\in kP_m\cap L_m$, $v\in V_m$, $A=(a_1,\dots,a_m)\in\{0,1,2\}^m$. Denote $D_i(A)=S_i(x,A)-S_i(x-v,A)$.
\begin{itemize}
\item If $a_1+\dots+a_m=5$, $a_m=2$, then $D_1(A)\in\{2,5\}$.
\item If $a_1+\dots+a_m=2$, then $D_1(A)\in\{2,5\}$. Moreover, if $v=v(0)$ or $a_m=2$ then $D_1(A)=2$. If $a_m=1,a_j=1$ and $v=v(m,j)$, then also $D_1(A)=2$.
\item If $a_1+\dots+a_m=2$ then $D_2(A)\in\{2,5\}$. Moreover, if $v=v(0)$ or $a_m=2$ then $D_2(A)=2$.  If $a_m=1$ and $v=v(m,j)$ for some $j$, then also $D_2(A)=2$.
\end{itemize}
\end{Lemma}
\begin{proof}
We check for every possible $v$ and $A$ how many summands in $S_i(x,A)$ will decrease when we subtract $v$. We see that it will be always decreased by 2 or 5.
\end{proof}

\begin{Remark}
We could list all cases when $D_i(A)=2$, but since it would be a long list, we just list those which will be important later in the proof.
\end{Remark}

Now we consider the following special case:

\begin{Proposition}\label{0} Let $x\in kP_m\cap L_m$ satisfying conditions $(1)-(3)$. Suppose that $0$ is also the most frequent element in the $m$-th multiset from $G$-presentation of $x$. Then $x-v(0)\in (k-1)P_m$.
\end{Proposition}
\begin{proof}
Obviously, every multiset from $G$-presentation of $x$ contains $0$, so $v(0)$ is $x$-good. Inequalities for $n$-tuples with $a_1+\dots+a_m\ge 5$ hold for $x-v$ by Proposition \ref{big}, because when $0$ is also most frequent in the $m$-th multiset there is no special case when $a_1+\dots+a_m=5$. Inequalities for $a_1+\dots+a_m=2$ hold by Lemma \ref{small}, since we are subtracting vertex $v(0)$.
It follows that $x-v(0)\in (k-1)P_m$.
\end{proof}

This result implies that we can assume that for $x\in kP_m\cap L_m$ satisfying $(1),(2),(3)$ also holds the following condition:
\begin{equation}
\begin{aligned}
&\text{There exists no } h\in H_m\text{ such that the following conditions holds:}
\\ &0\text{ is the most frequent element in all multisests from } G \text{-presentation of }hx.
\end{aligned}
\end{equation}

We show that in our case the point $x$ cannot lie on many facets.

\begin{Proposition}\label{A5}
Let $x\in kP_m\cap L_m$ satisfy $(1)-(4)$, $A=(a_1,\dots,a_m)\in\{0,1,2\}^m$ with $a_1+\dots+a_m=5$ and $a_m=2$. Then:
\begin{itemize}
\item[a)]$x$ does not belong to facet $F_1(A)$ , i.e. $S_1(x,A)>2k$, for all such $A$ and $x$.
\item[b)] $S_1(x-v,A)\ge 2(k-1)$, for all $v\in V_m$.
\end{itemize}
\end{Proposition}

\begin{proof}
We prove part a) by contradiction: Suppose that we have an equality for some $A$. Since $a_m=2$ there are two cases regarding $a_1,\dots,a_{m-1}$: either three of them are equal to 1 or one of them is equal to 1 and other to 2. By acting with suitable $\sigma\in S_m$ we get to the case where either $a_1=a_2=a_3=1$ or $a_1=1,a_2=2$. Consider the first case. Then

\begin{align*}
S_1(x,A)&\ge x_0^1+2x_1^1+x_0^2+2x_1^2+x_0^3+2x_1^3+2x_0^m+x_2^m\\
&\ge \frac 12(x_0^1+x_1^1+x_2^1)+\frac 12(x_0^2+x_1^2+x_2^2)+(x_0^3-x_1^m)+(x_0^m+x_1^m+x_2^m)\\
&\ge \frac k2+\frac k2+0+k=2k.
\end{align*}

We use the ordering of multisets (condition $(1)$) for $(x_0^3-x_1^m)\ge 0$. To get an equality we must have equality everywhere. In particular, we get $x_0^1=x_2^1$, $x_1^1=x_1^3=x_0^m=0$ and $x_0^3=x_1^m$. From condition $(2) $ we obtain that $\{x_0^1,x_1^1,x_2^1\}=\{1/2,1/2,0\}\succeq \{x_0^m,x_1^m,x_2^m\}=\{0,x_1^m,x_2^m\}$, which means $x_1^m=x_2^m=1/2$. By acting with $(2,0,0,\dots,0,1)$ we get to the situation where $0$ is the most frequent element in all multisets, which is a contradiction with $(4)$.

The case $a_1=1,a_2=2$ is analogous:
\begin{align*}
S_1(x,A)&\ge x_0^1+2x_1^1+2x_0^2+x_2^2+2x_0^m+x_2^m\\
&\ge \frac 12(x_0^1+x_1^1+x_2^1)+\frac 12(x_0^2+x_1^2+x_2^2)+(x_0^2-x_1^m)+(x_0^m+x_1^m+x_2^m)\\
&\ge \frac k2+\frac k2+0+k=2k.
\end{align*}

Similarly, we get that $x_1^m=x_2^m=x_0^1=x_2^1=1/2$ and we can act with $(2,0,0,\dots,0,1)$ to obtain the same situation, where $0$ is the most frequent element in all multisets, contradicting $(4)$.

We continue with the proof of part b). Part a) together with Lemma \ref{2} implies that $S_i(x,A)\ge k+2$. Consequently, Lemma \ref{small} implies $S_1(x-v,A)\ge 2(k-1)$, for any $v\in V_m$.
\end{proof}

\begin{Proposition}\label{x_n>0}
Let $x\in kP_m\cap L_m$ satisfying $(1)-(4)$ and $x_0^m>0$. Then $x-v(0)\in (k-1)P_m$.
\end{Proposition}
\begin{proof}
Clearly, $v(0)$ is $x$-good. Inequalities for $A$ with $a_1+\dots+a_m\ge 5$ hold by Propositions \ref{big} and \ref{A5}. For $a_1+\dots+a_m=2$ we have $S_i(x,A)\ge 2k$, then by Lemma \ref{small} we get $S_i(x-v(0),A)\ge 2(k-1)$. Therefore, all inequalities hold and $x-v(0)\in (k-1)P_m$.
\end{proof}

This means we are left with the special case $x_0^m=0$. In this case we cannot subtract $v(0)$ since it is not an $x$-good vertex. We will use some other vertex from $V_m$. We now have to also check facets for $m$-tuples $A$ with $a_1+\dots+a_m=2$. Again, we show that our assumptions imply that $x$ doe not lie on some of these facets:

\begin{Proposition}\label{nofacet2}
Let $x\in kP_m\cap L_m$ satisfy $(1)-(4)$, $x_0^m=0$ and  $A=(a_1,\dots,a_m)\in\{0,1,2\}^m$ with $a_1+\dots+a_m=2$ and $a_m=0$. Then $x$ does not belong to the facet $F_i(A)$.
\end{Proposition}
\begin{proof}
We have two options: either two of the numbers $a_j$ are 1 or one of the numbers is 2. We can assume without loss of generality $a_1=a_2=1$ in the first case and $a_1=2$ in the second. Also, without loss of generality, we can assume $i=1$.

We start with the first case and compute $S_1(x,A)$:
\begin{align*}
S_1(x,A)&\ge x_0^1+2x_1^1+x_0^1+2x_1^1+x_1^m+2x_2^m\\
&\ge \frac12(x_0^1+x_1^1+x_2^1)+\frac 12(x_0^2+x_1^2+x_2^2)+(x_1^m+x_2^m)\ge \frac k2+\frac k2+k=2k.
\end{align*}

To get an equality we must have $x_2^m=0$ and $x_1^m=k$. Condition $(1)$ and $(2)$ then gives us $x_0^1=k$ which is contradiction with the condition $(3)$.

The case $a_1=2$ is similar:
\begin{align*}
S_1(x,A)&\ge 2x_0^1+x_1^1+x_1^m+2x_2^m\\
&\ge(x_0^2+x_1^2+x_2^2)+(x_1^m+x_2^m)\ge k+k=2k.
\end{align*}

Again, we must have $x_2^m=0$ and $x_1^m=k$ which is impossible by the same reason.
\end{proof}

For the point $x$ satisfying $(1)-(4)$ we can also assume that $x_1^m\ge x_2^m$. We can achieve that by acting with suitable $\varphi\in\Aut(\ZZ_3)$. Now we solve the case when $x$ does not lie on any facet $F_1(A)$ with $a_1+\dots+a_m=2,a_m=1$ by adding one assumption:

\begin{Proposition}\label{0nof}
Let $x\in kP_m\cap L_m$ satisfy $(1)-(4)$, $x_0^m=0,x_1^m\ge x_2^m$ and $x$ does not belong to any facet $F_1(A)$ with $a_1+\dots+a_m=2$, $a_m=1$. Moreover, suppose that there exists an index $j<m$ such that $x_2^j>0$. Then  $x-v(m,j)\in (k-1)P_m$.
\end{Proposition}

\begin{proof}
Clearly, the vertex $v(m,j)$ is $x$-good. Then Propositions \ref{big} and \ref{A5} and Lemma \ref{small} imply that the inequalities for $m$-tuple $A$ with $a_1+\dots a_m\ge 5$ hold. For $a_1+\dots+a_m=2$ Lemma \ref{small} and Lemma \ref{nofacet2} implies $S_i(x-v,A)\ge S_i(x,A)-5\ge 2(k-1)$, where $S_i(x,A)\ge 2k+3$ follows from Lemma \ref{2}.

 Therefore, inequalities for every $m$-tuple $A$ hold, which proves our result.
\end{proof}

If $x$ belongs to facet $F_1(A)$ with $a_1+\dots+a_m=2,a_m=1$ we prove that it belongs only to one such facet and that we can as well subtract some vertex $v(m,j)$:

\begin{Proposition}\label{0f}
Let $x\in kP_m\cap L_m$ satisfy $(1)-(4)$, $x_0^m=0$, $x_1^m\ge x_2^m$ and $x$ belongs to some facet $F_1(A)$ for $a_1+\dots+a_m=2$, $a_m=1$. Moreover, suppose that there exists an index $j<m$ such that $x_2^j>0$. Then
\begin{itemize}
\item[a)] $x$ belongs to only one such facet.
\item[b)] There exists a vertex $v\in V_m$ such that $x-v\in (k-1)P_m$.
\end{itemize}
\end{Proposition}
\begin{proof}

To prove part a) suppose that $x$ belongs to two such facets. By acting with suitable permutation from $\mathbb S_m$ we can get to the situation where these facets are $F_1((1,0,\dots,0,1))$ and $F_1(0,1,0,\dots,0,1)$.

Then we compute the following sum:

\begin{align*}
S_1(x,(1,0,\dots,0,1))+S_1(x,(0,1,0,\dots,0,1))
\end{align*}
\begin{align*}
&\ge (x_0^1+3x_1^1+2x_2^1)+(x_0^2+3x_1^2+2x_2^2)+(4x_1^m)\\
&\ge k+k+2(x_1^m+x_2^m)=4k.
\end{align*}

We know that there must be an equality and therefore $x_1^1=x_2^1=0$ which is a contradiction with condition $(3)$.

For part b) we again suppose without loss of generality that $x\in F_1((1,0,\dots,0,1)$. We show  by contradiction that $x_2^1>0$. Suppose that $x_2^1=0$. Then
\begin{align*}
S_1(x,(1,0,\dots,0,1))&\ge (x_0^1+2x_1^1)+2x_1^m\\
&\ge (x_0^1+x_1^1) +x_1^1+(x_1^m+x_2^m)\ge k+0+k=2k
\end{align*}

To get an equality we must have $x_1^1=0$ and $x_0^1=k$ which is contradiction with condition $(3)$. This implies that vertex $v(m,1)$ is $x$-good. We claim that $x-v(m,1)$ satisfy all inequalities and hence belongs to $(k-1)P_m$.

Propositions \ref{big} and \ref{A5} and Lemma \ref{small} imply that the inequalities for $m$-tuple $A$ with $a_1+\dots a_m\ge 5$ hold. Lemma \ref{small} together with Lemma \ref{nofacet2} and part a) gives the inequalities for $a_1+\dots a_m=2$.

\end{proof}

We are left only with one special case. That is the case when $x$ does not belong to facet $F_1(A)$ with $a_1+\dots+a_m=2,a_m=1$ and $x_2^j=0$ for all $1\le j\le m-1$.

\begin{Proposition}\label{mostspecific}
Let $x\in kP_n\cap L_n$ satisfy $(1)-(4)$, $x_0^m=0,x_1^m\ge x_2^m$ and $x$ does not belong to any $F_1(A)$ for $a_1+\dots+a_m=2$, $a_m=1$. Moreover, suppose that $x_2^j=0$ for all $1\le j\le m-1$. Then  $x-v(1,m)\in (k-1)P_m$.

\end{Proposition}

\begin{proof}
We notice that $x_1^1$ and $x_2^m$ are non-zero due to $(3)$. This means that the vertex $v(1,m)$ is $x$-good. We need to check the inequalities. As in the previous cases everything is covered by Propositions \ref{big} and \ref{A5} and Lemmas \ref{small} and \ref{nofacet2} except the case of facets $F_2(A)$ with $a_1+\dots+a_m=2,a_m=1$. If we can prove that $x$ does not lie on any such facet then we can conclude by Lemma \ref{small}.

Suppose by contradiction that (without loss of generality) $x \in F_2((1,0,\dots,0,1))$.
  Then
\begin{align*}
k=S_2(x,(1,0,\dots,0,1))&=x_0^1+2(x_2^1+\dots+x_1^{m-1})+2x_2^m\\
&\ge (x_1^1+\dots+x_1^{m-1}+x_2^m)\\
&=S_1(x,(0,0,\dots,0,2)\ge k.
\end{align*}

To get the equality we must have $x_2^m=0$ but we have already shown that it is not the case.

\end{proof}

Now we can sum up and prove Theorem \ref{normalz3}.

\begin{proof} (of Theorem \ref{normalz3})

Consider point $x\in kP_m\cap L_m$ for some positive integer $k$.  To prove normality of $P_m$ it is sufficient for $k\ge 2$ to prove that there exists a vertex $v$ of $P_m$ such that $x-v\in (k-1)P_m$. Also it is sufficient to consider only points $x$ which satisfy $(1)-(3)$. Then the existence of such $v$ is implied by Propositions \ref{0}, \ref{x_n>0}, \ref{0nof}, \ref{0f} and \ref{mostspecific}.
\end{proof}

 \section*{Acknowledgements}
 The authors would like to thank very much Mateusz Micha\l ek for introducing to the topic and giving helpful advice. Both authors were supported by the Max Planck Institute for Mathematics in the Sciences. They used Polymake \cite{polymake} for their working examples.

\end{document}